\documentclass{article}

\usepackage[utf8]{inputenc}
\usepackage{amsmath,amssymb,dsfont,amsfonts}
\usepackage{amsthm}
\usepackage{mathrsfs}
\usepackage{graphics}
\usepackage{graphicx}
\usepackage{color}

\newtheorem{example}{Example}

\newtheorem{theorem}{Theorem}[section]
\newtheorem{proposition}[theorem]{Proposition}
\newtheorem{corollary}[theorem]{Corollary}
\newtheorem{remark}[theorem]{Remark}
\newtheorem{definition}[theorem]{Definition}
\newtheorem{lemma}[theorem]{Lemma}

\usepackage[bottom=2cm,top=2.5cm,left=2.5cm,right=2.5cm]{geometry}

\title{Goldbach and Alexandroff  Shake Hands}

\author{
Lourival Lima\footnote{Mathematics Department, Escuela Superior
Politécnica del Litoral, ESPOL. Vía Perimetral 5, Guayaquil,
Ecuador. E-mail: lourodri@espol.edu.ec.
Research supported by ESPOL.}
\ \ and\ \
Jorge Vielma\footnote{Mathematics Department, Escuela Superior
Politécnica del Litoral, ESPOL. Vía Perimetral 5, Guayaquil,
Ecuador. E-mail: jevielma@espol.edu.ec.
Research supported by ESPOL.}
}

\date{}

\begin{document}

\maketitle

\begin{abstract}
We introduce a relational Alexandroff topology associated with the strong
Goldbach conjecture on the set of even integers greater than or equal to
six. The construction is induced by a multivalued map encoding Goldbach
representations and leads naturally to a notion of relational reachability.
For arbitrary binary relations, we show that singleton closures coincide
with reachable sets and that the common intersection of all singleton
closures consists exactly of the globally reachable points. In the Goldbach
setting, this yields a topological characterization of the conjecture in
terms of the distinguished point \(6\). In particular, the strong Goldbach
conjecture is equivalent to the associated relational space being non-\(wR_0\)
and connected. We also obtain compactness and supercompactness criteria in
terms of the set of possible Goldbach exceptions.
\end{abstract}


\vspace{1cm}

\section{Introduction}

The strong Goldbach conjecture is one of the most classical open problems
in number theory. It asserts that every even integer greater than two can
be represented as the sum of two prime numbers. Despite its elementary
formulation and extensive computational verification, the conjecture has
resisted a complete proof. In particular, large-scale computations have
verified the conjecture for all even integers up to \(4\cdot 10^{18}\)
\cite{OliveiraHerzogPardi2014}, but no argument is known that covers all
even integers.

Topological methods have proved useful in the study of discrete
mathematical structures, especially when algebraic, combinatorial, or
dynamical information can be encoded through specialization preorders,
reachability relations, or Alexandroff topologies. Functional Alexandroff
spaces provide a natural framework for this type of construction
\cite{ShiraziGolestani2011,Pajoohesh2024}, while the interplay between
preorders and topologies has been studied from several perspectives,
including lattice-theoretic aspects \cite{Uzcategui2003}. Closely related
ideas also arise in the theory of primal spaces and their continuous maps
\cite{EchiTurki2019,MejiasVielmaAponteLima2025}.

Arithmetic dynamics provide a particularly suggestive setting for such
topological constructions. In the context of the Ulam--Kakutani--Collatz
conjecture, Vielma and Guale introduced a primal topology associated with
the Collatz map and obtained a characterization of the conjecture in terms
of the weak separation property \(wR_0\)
\cite{VielmaGuale2019}. More recently, the same topological structure has
been investigated from an algebraic point of view through the semiring of
open sets \cite{GualeVielma2026}. These works suggest that global arithmetic
behaviour may sometimes be captured by global topological properties of a
space naturally induced by the underlying discrete structure.

For Goldbach, the situation is structurally different. The conjecture is
not generated by the iteration of a single map, but by the existence of
prime decompositions. For an even integer \(m\geq 6\), define
\[
S(m)=\{p\in\mathbb P:m-p\in\mathbb P\}.
\]
Thus \(S(m)\) records all primes that occur in Goldbach representations of
\(m\). Rather than selecting a particular decomposition, we retain all
representations simultaneously and introduce the multivalued map
\[
F_G(m)=\{p+3:p\in S(m)\}.
\]
This gives rise to a binary relation \(\mathcal R_G\) on
\[
E_6=\{6,8,10,\ldots\},
\]
defined by
\[
m\mathcal R_G n
\quad\text{whenever}\quad
n\in F_G(m).
\]
The use of the prime \(3\) is not merely notational. If
\(p\in S(m)\), then \(m\mathcal R_G(p+3)\), while
the number \(p+3\) has \(3\) as a Goldbach prime, which yields a second
relational step to \(6\). Hence every
Goldbach representation generates a path of length at most two ending at
the distinguished point \(6\).

Given a binary relation \(R\) on a set \(X\), for each \(A\subseteq X\) we write
\[
R^{*}(A)=\{x\in X:\text{ there exists }y\in A\text{ such that }xRy\},
\]
and consider the topology
\[
\tau_R=
\{U\subseteq X:R^{*}(U)\subseteq U\}.
\]
This topology is Alexandroff, and its closed subsets are precisely the
subsets that are invariant under forward \(R\)-transitions. As a result,
the closure of a singleton admits a direct relational interpretation:
\[
\overline{\{x\}}^{\,\tau_R}
=
\{y\in X:x\leadsto_R y\},
\]
where \(x\leadsto_R y\) denotes reachability through a finite \(R\)-chain.
This observation motivates the study of the set
\[
K_R(X)=\bigcap_{x\in X}\overline{\{x\}},
\]
which consists exactly of those points that are reachable from every point
of the space. In the terminology of Di Maio \cite{DiMaio1985}, the
nonemptiness of this intersection is equivalent to the failure of the
\(wR_0\) property.

The main purpose of this paper is to show that this relational framework
provides a natural Alexandroff-topological formulation of the strong
Goldbach conjecture. For the Goldbach relation \(\mathcal R_G\), we prove
that the conjecture holds precisely when
\[
\bigcap_{m\in E_6}\overline{\{m\}}=\{6\}.
\]
Thus the conjecture is characterized by the failure of the \(wR_0\)
property in the associated Goldbach relational space. We also establish a
connectedness characterization and investigate compactness and
supercompactness. In particular, connectedness detects the full Goldbach
conjecture, whereas compactness is equivalent to the set of possible
Goldbach exceptions being finite.

Graph-theoretic formulations of Goldbach have also been considered in the
literature. In particular, Das, Ghosh, Ghosh and Sen introduced Goldbach
graphs and related the conjecture to graph connectedness
\cite{DasGhoshGhoshSen2021}. The present approach is different in nature:
the primary object is not an auxiliary graph but an Alexandroff topology
defined directly on the even integers, and the central objects are
singleton closures, relational reachability, separation properties and
compactness-type phenomena.

The paper is organized as follows. We first develop the general theory of
relational Alexandroff topologies and establish the correspondence between
closures and reachability. We then introduce the global relational core and
its relation with the \(wR_0\) property. After defining the Goldbach
relation and its associated topology, we prove the main topological
characterization of the strong Goldbach conjecture and finaly, we study the  structural properties of the Goldbach space, including
connectedness, compactness and supercompactness.

\section{Relational Alexandroff topologies}

This section develops the general topological framework that will be used
throughout the paper. Topologies induced by functions and discrete
dynamical structures have been studied in the setting of functional
Alexandroff spaces and primal spaces
\cite{Echi2012,ShiraziGolestani2011,Pajoohesh2024}. Here we work with an
arbitrary binary relation and associate with it an Alexandroff topology
whose open and closed sets reflect relational transitions. The main goal is
to establish a precise connection between topological closure and
relational reachability, which will later allow arithmetic information to
be translated into topological terms.

\begin{definition}
Let $X$ be a nonempty set and let
\[
R\subseteq X\times X
\]
be a binary relation. For $A\subseteq X$, define the relational predecessor operator
\[
R^{*}(A)
=
\{x\in X:\text{ there exists }y\in A\text{ such that }xRy\}.
\]
The family
\[
\tau_R
=
\left\{
U\subseteq X:R^{*}(U)\subseteq U
\right\}
\]
will be called the \emph{relational topology induced by $R$}.
\end{definition}

\bigskip

The preceding definition is meaningful only if the family \(\tau_R\) actually
satisfies the axioms of a topology. The next result shows that no additional
assumptions on the relation are required and, moreover, that the resulting
topology is automatically Alexandroff.

\begin{proposition}
\label{prop:relational-topology}
For every binary relation $R$ on $X$, the family $\tau_R$ is an
Alexandroff topology on $X$.
\end{proposition}

\begin{proof}
Clearly, $R^{*}(\varnothing)=\varnothing$, and hence
$\varnothing\in\tau_R$. Moreover, $R^{*}(X)\subseteq X$, so that $X\in\tau_R$. Let $\{U_i\}_{i\in I}$ be a family of elements of $\tau_R$. Since the relational predecessor operator preserves arbitrary unions,
\[
R^{*}\left(\bigcup_{i\in I}U_i\right)
=
\bigcup_{i\in I}R^{*}(U_i).
\]
Therefore,
\[
R^{*}\left(\bigcup_{i\in I}U_i\right)
\subseteq
\bigcup_{i\in I}U_i,
\]
and hence $\displaystyle \bigcup_{i\in I}U_i\in\tau_R$. For intersections, we have
\[
R^{*}\left(\bigcap_{i\in I}U_i\right)
\subseteq
\bigcap_{i\in I}R^{*}(U_i).
\]
Since each $U_i$ belongs to $\tau_R$, $R^{*}(U_i)\subseteq U_i$. Consequently,
\[
R^{*}\left(\bigcap_{i\in I}U_i\right)
\subseteq
\bigcap_{i\in I}U_i.
\]
Thus, $\bigcap_{i\in I}U_i\in\tau_R$.
\end{proof}

\begin{definition}
Let $R$ be a binary relation on $X$. We say that $y$ is
\emph{$R$-reachable from $x$}, and write $x\leadsto_R y$, if either $x=y$, or there exist $n\geq1$ and points $x=x_0,x_1,\ldots,x_n=y$ such that $x_{j-1}Rx_j,
 \ 
j=1,\ldots,n$. Equivalently, $\leadsto_R$ is the reflexive and transitive closure of
$R$.
\end{definition}

\bigskip

The definition of \(\tau_R\) is expressed through the predecessor operator,
whereas the closed sets admit a more direct dynamical interpretation. The
following proposition identifies closedness with stability under forward
\(R\)-transitions and will be used repeatedly in the sequel.

\begin{proposition}
\label{prop:closed-forward}
A subset $C\subseteq X$ is closed in $(X,\tau_R)$ if and only if it is
forward invariant under $R$; that is,
\[
x\in C,\quad xRy
\quad \text{then}\quad
y\in C.
\]
\end{proposition}

\begin{proof}
The set $C$ is closed if and only if $X\setminus C$ is open. Hence, $C\text{ is closed, if and only if}
 \ \
R^{*}(X\setminus C)\subseteq X\setminus C$. Suppose $x\in C$ and $xRy$. If $y\notin C$, then $y\in X\setminus C$, and therefore $x\in R^{*}(X\setminus C)$. This inclusion would imply $x\in X\setminus C$, contradicting $x\in C$. Hence $y\in C$. Conversely, suppose that $C$ is forward invariant. If $x\in R^{*}(X\setminus C)$, then there exists $y\notin C$ such that $xRy$. If $x$ belonged to
$C$, forward invariance would imply $y\in C$, a contradiction.
Therefore $x\notin C$, and hence
\[
R^{*}(X\setminus C)\subseteq X\setminus C.
\]
Thus $X\setminus C$ is open and $C$ is closed.
\end{proof}

\bigskip

Once closed sets are understood as forward invariant sets, singleton
closures can be described purely in terms of the relation. This is the key
link between the topology and the underlying reachability structure, and it
will later allow Goldbach representations to be read through closures.

\begin{theorem}
\label{thm:closure-reachability}
Let $R$ be a binary relation on $X$. Then, for every $x\in X$,
\[
\overline{\{x\}}^{\,\tau_R}
=
\{y\in X:x\leadsto_R y\}.
\]
\end{theorem}

\begin{proof}
Fix $x\in X$ and define
\[
\mathcal O_R(x)
=
\{y\in X:x\leadsto_R y\}.
\]
Since reachability is reflexive, $x\in\mathcal O_R(x)$. We claim that $\mathcal O_R(x)$ is closed. Indeed, suppose that $y\in\mathcal O_R(x)$ and $yRz$. Since $x\leadsto_Ry$, appending the relation $yRz$
gives $x\leadsto_R z$. Thus $z\in\mathcal O_R(x)$. By Proposition~\ref{prop:closed-forward},
$\mathcal O_R(x)$ is closed. Now let $C$ be any closed subset of $X$ containing $x$. Since $C$ is
forward invariant, whenever
\[
x=x_0Rx_1R\cdots Rx_n=y,
\]
an induction shows that
\[
x_0,x_1,\ldots,x_n\in C.
\]
Hence $y\in C$. Therefore $\mathcal O_R(x)\subseteq C$. Thus $\mathcal O_R(x)$ is the smallest closed subset of $X$
containing $x$, and consequently
\[
\overline{\{x\}}
=
\mathcal O_R(x)
=
\{y\in X:x\leadsto_Ry\}.
\]
\end{proof}

\section{The global relational core and the $wR_0$ property}

The reachability description of singleton closures suggests considering
their common intersection as a global invariant of the relational space.
This viewpoint is closely related to weak separation properties, in
particular the \(wR_0\) condition introduced by Di Maio
\cite{DiMaio1985} and later used in primal-topological formulations of
arithmetic problems \cite{VielmaGuale2019}. In this section we introduce
the global relational core and show that it consists precisely of the
points that can be reached from every point of the space.
\begin{definition}
Let $(X,\tau_R)$ be the relational topological space induced by $R$.
The set
\[
K_R(X)
=
\bigcap_{x\in X}\overline{\{x\}}
\]
will be called the \emph{global relational core} of $(X,R)$.
\end{definition}

\begin{definition}
A topological space $X$ is said to satisfy the $wR_0$ property if
\[
\bigcap_{x\in X}\overline{\{x\}}=\varnothing.
\]
\end{definition}

The global relational core is defined topologically, but the previous
closure--reachability theorem suggests that it should have a direct
relational description. The next theorem confirms this and at the same time
connects the core with the \(wR_0\) separation property.

\begin{theorem}[Global core characterization]
\label{thm:global-core}
Let $R$ be a binary relation on a nonempty set $X$. Then
\[
K_R(X)
=
\left\{
a\in X:
x\leadsto_R a
\text{ for every }x\in X
\right\}.
\]
Consequently, the following statements are equivalent:
\begin{enumerate}
    \item $K_R(X)\neq\varnothing$;
    \item there exists $a\in X$ such that
    \[
    x\leadsto_R a
    \qquad
    \text{for every }x\in X;
    \]
    \item $(X,\tau_R)$ is not $wR_0$.
\end{enumerate}
\end{theorem}

\begin{proof}
By Theorem~\ref{thm:closure-reachability}, $\overline{\{x\}}
=
\{y\in X:x\leadsto_R y\}$. Therefore,
\begin{align*}
K_R(X)
&=
\bigcap_{x\in X}\overline{\{x\}}
\\
&=
\bigcap_{x\in X}
\{y\in X:x\leadsto_R y\}
\\
&=
\left\{
a\in X:
x\leadsto_Ra
\text{ for every }x\in X
\right\}.
\end{align*}
Hence $K_R(X)$ is nonempty if and only if there exists a point that
is reachable from every point of $X$. Finally, by the definition of
the $wR_0$ property,
\[
(X,\tau_R)\text{ is not }wR_0
 \ \text{if and only if} \ 
K_R(X)\neq\varnothing.
\]
The result follows.
\end{proof}

\bigskip

The previous characterization also imposes a strong internal structure on
the global core. Since every point of the ambient space reaches each core
point, core points necessarily communicate with one another, and
reachability from a core point cannot leave the core.

\begin{proposition}
\label{prop:core-strongly-connected}
Suppose that $K_R(X)\neq\varnothing$. Then any two points of
$K_R(X)$ are mutually reachable. Moreover, $K_R(X)$ is forward
invariant under reachability.
\end{proposition}

\begin{proof}
Let $a,b\in K_R(X)$. Since $b\in K_R(X)$, the definition of the
global core implies that $x\leadsto_R b
 \  \ 
\text{for every }x\in X$. Taking $x=a$, we obtain $a\leadsto_R b$. Similarly, $b\leadsto_R a$. Now suppose that $a\in K_R(X)$ and $a\leadsto_Rz$. For every $x\in X$, since $a\in K_R(X)$, $x\leadsto_Ra$. By transitivity of reachability, $x\leadsto_Ra\leadsto_Rz$, and hence $x\leadsto_Rz$. Thus $z\in K_R(X)$.
\end{proof}

\section{The Goldbach relational topology}

We now specialize the general relational construction to the strong
Goldbach conjecture. Unlike functional primal spaces, which arise from the
iteration of a single map \cite{Echi2012,ShiraziGolestani2011}, the Goldbach
problem naturally involves several simultaneous prime decompositions. Our
construction therefore uses a multivalued map and the binary relation it
induces on the even integers. This relational viewpoint provides the
Alexandroff space that will serve as the basic topological model for the
remainder of the paper. Let $\mathbb P$ denote the set of prime numbers, and define
\[
E_6
=
\{2n:n\geq3\}
=
\{6,8,10,12,\ldots\}.
\]
The case $4=2+2$ is immediate and will therefore be treated
separately when necessary.

\bigskip

\begin{definition}
For each $m\in E_6$, define its \emph{Goldbach prime set} by
\[
S(m)
=
\{p\in\mathbb P:m-p\in\mathbb P\}.
\]
\end{definition}

\bigskip

Thus $S(m)\neq\varnothing$ if and only if $m$ admits a representation as the sum of two prime
numbers. For instance, $S(10)=\{3,5,7\}$, since $10=3+7=5+5=7+3$.

\bigskip

\begin{definition}
For $m\in E_6$, define the multivalued Goldbach map $F_G(m)
=
\{p+3:p\in S(m)\}$. The \emph{Goldbach relation} $\mathcal R_G$ on $E_6$ is defined by
\[
m\mathcal R_G n
 \ \  \text{if and only if,} \ \ 
n\in F_G(m).
\]
Equivalently, $m\mathcal R_G n$
if and only if there exists $p\in S(m)$
such that $n=p+3$.
\end{definition}

\begin{remark}
If $m\in E_6$ and $p\in S(m)$, then $p$ must be an odd prime. Indeed,
$p=2$ would imply that $m-2$ is an even integer greater than or equal
to $4$, and hence cannot be prime. Consequently, $p+3\in E_6$. Thus
\[
F_G(m)\subseteq E_6
\]
for every $m\in E_6$.
\end{remark}

\bigskip

\begin{example}
Since $S(10)=\{3,5,7\}$, we have $F_G(10)
=
\{6,8,10\}$. Hence
\[
10\mathcal R_G6,
\qquad
10\mathcal R_G8,
\qquad
10\mathcal R_G10.
\]
\end{example}

\bigskip

\begin{definition}
For \(A\subseteq E_6\), let
\[
\mathcal R_G^{*}(A)
=
\{m\in E_6:\text{ there exists }n\in A\text{ such that }m\mathcal R_G n\}.
\]
The \emph{Goldbach relational topology} is the topology
\[
\tau_G
=
\left\{
U\subseteq E_6:
\mathcal R_G^{*}(U)\subseteq U
\right\}.
\]
The corresponding topological space will be denoted by $(E_6,\tau_G)$.
\end{definition}

\bigskip

The preceding construction places the Goldbach relation within the general
relational framework developed in Section~2. Since the associated topology
is defined through the same predecessor condition, the general results on
relational topologies apply directly to \(\mathcal R_G\). In particular, the
Goldbach relational topology inherits the Alexandroff property established
above.
\begin{corollary}
The Goldbach relational topology $\tau_G$ is an Alexandroff topology
on $E_6$.
\end{corollary}

The structure of the Goldbach relation is governed by a distinguished
point, namely \(6\). In functional Alexandroff and primal spaces, orbit and
reachability structures play a central role in describing closures
\cite{EchiTurki2019,MejiasVielmaAponteLima2025}. In the present relational
setting, we obtain an analogous phenomenon: every Goldbach representation
produces a short path ending at \(6\). 
\bigskip

The essential feature of the chosen relation is that a single Goldbach
representation creates a short route toward the distinguished point \(6\).
The next lemma makes this mechanism precise and is the arithmetic input
behind all subsequent topological characterizations.

\begin{lemma}
\label{lem:goldbach-to-six}
Let $m\in E_6$. If $S(m)\neq\varnothing$, then $m\leadsto_{\mathcal R_G}6$. More precisely, $6$ can be reached from $m$ in at most two
$\mathcal R_G$-steps.
\end{lemma}

\begin{proof}
Assume that $S(m)\neq\varnothing$. Choose $p\in S(m)$. By the definition of $F_G$, $p+3\in F_G(m)$, and hence $m\mathcal R_G(p+3)$. Let $n=p+3$. Since $p$ and $3$ are prime, $n-3=p\in\mathbb P$, and therefore $3\in S(n)$. It follows that $3+3=6\in F_G(n)$, and hence $n\mathcal R_G6$. Consequently,
\[
m\mathcal R_G(p+3)\mathcal R_G6.
\]
Therefore
\[
m\leadsto_{\mathcal R_G}6.
\]
\end{proof}

The reachability statement has an immediate topological interpretation.
Indeed, the general closure--reachability theorem converts the finite
Goldbach path into membership of \(6\) in a singleton closure. This simple
observation will be used in the proof of the main theorem.

\begin{corollary}
\label{cor:six-closure}
If $S(m)\neq\varnothing$, then $6\in\overline{\{m\}}^{\,\tau_G}$.
\end{corollary}

\begin{proof}
By Lemma~\ref{lem:goldbach-to-six}, $m\leadsto_{\mathcal R_G}6$. The conclusion follows from
Theorem~\ref{thm:closure-reachability}.
\end{proof}

\bigskip

The point $6$ plays a distinguished role in the Goldbach relational
space. To identify a common point of all singleton closures, it is also necessary
to understand the closure of the distinguished point itself. The next lemma
shows that \(6\) has no nontrivial forward reachability and hence forms a
closed singleton in the Goldbach relational space.

\begin{lemma}
\label{lem:closure-six}
In $(E_6,\tau_G)$, $\overline{\{6\}}=\{6\}$.
\end{lemma}

\begin{proof}
Since $6=3+3$, we have $S(6)=\{3\}$. Hence $F_G(6)=\{3+3\}=\{6\}$. Thus the only $\mathcal R_G$-transition starting from $6$ is $6\mathcal R_G6$. Consequently, the only point reachable from $6$ is $6$ itself.
By Theorem~\ref{thm:closure-reachability}, $\overline{\{6\}}=\{6\}$.
\end{proof}

\bigskip

The opposite arithmetic situation is equally simple from the relational
point of view. If an even integer admits no Goldbach representation, then
the multivalued map has no value at that point, so no nontrivial relational
path can start there. This yields the following closure description.

\begin{proposition}
\label{prop:failure-goldbach-isolated}
Let $m\in E_6$. If $S(m)=\varnothing$, then $\overline{\{m\}}=\{m\}$.
\end{proposition}

\begin{proof}
If $S(m)=\varnothing$, then $F_G(m)=\varnothing$. Hence there is no nontrivial $\mathcal R_G$-transition starting from
$m$. Therefore the only point reachable from $m$ is $m$ itself.
Theorem~\ref{thm:closure-reachability} yields $\overline{\{m\}}=\{m\}$.
\end{proof}

\section{A $wR_0$ characterization of the Goldbach conjecture}

We are now in a position to establish the main result. Throughout this section, by the Goldbach conjecture we mean the strong
Goldbach conjecture.

The preceding lemmas isolate the exact role of the point \(6\): every
Goldbach-admissible even integer reaches \(6\), while \(6\) itself has
singleton closure. These two facts allow the arithmetic conjecture to be
encoded by the common intersection of singleton closures.

\begin{theorem}[Main characterization theorem]
\label{thm:main-goldbach}
The following statements are equivalent:
\begin{enumerate}
    \item[\rm (i)] The strong Goldbach conjecture is true.

    \item[\rm (ii)] $\bigcap_{m\in E_6}\overline{\{m\}}
    \neq\varnothing$.

    \item[\rm (iii)] $\bigcap_{m\in E_6}\overline{\{m\}}
    =
    \{6\}$.
    
    \item[\rm (iv)] The Goldbach relational space $(E_6,\tau_G)$ is not $wR_0$.
\end{enumerate}
\end{theorem}

\begin{proof}

Assume first that the strong Goldbach conjecture is true. Then every
$m\in E_6$ can be written as the sum of two primes. Equivalently, $S(m)\neq\varnothing
 \ 
\text{for every }m\in E_6$. By Corollary~\ref{cor:six-closure}, $6\in\overline{\{m\}}
 \ \
\text{for every }m\in E_6$. Hence $6\in
\bigcap_{m\in E_6}\overline{\{m\}}$. Therefore
\[
\{6\}
\subseteq
\bigcap_{m\in E_6}\overline{\{m\}}.
\]
On the other hand, since $6\in E_6$,
\[
\bigcap_{m\in E_6}\overline{\{m\}}
\subseteq
\overline{\{6\}}.
\]
By Lemma~\ref{lem:closure-six}, $\overline{\{6\}}=\{6\}$. Consequently,
\[
\bigcap_{m\in E_6}\overline{\{m\}}
=
\{6\}.
\]
Thus $\mathrm{(i)}\ \text{implies}\ \mathrm{(iii)}$. Clearly, $\mathrm{(iii)}\ \text{implies}\ \mathrm{(ii)}$. We now prove that $\mathrm{(ii)}\ \text{implies}\ \mathrm{(i)}$. Suppose that $\bigcap_{m\in E_6}\overline{\{m\}}
\neq\varnothing$. Since this intersection is contained in $\overline{\{6\}}$ and $\overline{\{6\}}=\{6\}$, we necessarily have $6\in\overline{\{m\}}
 \ \
\text{for every }m\in E_6$. Assume, by contradiction, that the strong Goldbach conjecture is
false. Since $4=2+2$, there must exist some $m_0\in E_6$ which cannot be expressed as the sum of two primes. Equivalently, $S(m_0)=\varnothing$. By Proposition~\ref{prop:failure-goldbach-isolated}, $\overline{\{m_0\}}=\{m_0\}$.
Since $m_0\neq6$, $6\notin\overline{\{m_0\}}$, contradicting $6\in
\bigcap_{m\in E_6}\overline{\{m\}}$. Hence every even integer greater than or equal to $6$ is a sum of two
primes. Together with $4=2+2$, this proves the strong Goldbach conjecture.
Finally, the last implication follows directly from the definition of the $wR_0$ property,

\end{proof}

The main theorem can be restated directly in terms of the global
relational core introduced earlier. This formulation emphasizes that the
Goldbach conjecture is equivalent to the existence of a globally reachable
point in the Goldbach relational space.

\begin{corollary}
\label{cor:global-core-goldbach}
The strong Goldbach conjecture is true if and only if the global
relational core of the Goldbach space is nonempty.
\end{corollary}

\begin{proof}
By definition,
\[
K_{\mathcal R_G}(E_6)
=
\bigcap_{m\in E_6}\overline{\{m\}}.
\]
The result follows immediately from
Theorem~\ref{thm:main-goldbach}.
\end{proof}

\begin{remark}
Theorem~\ref{thm:main-goldbach} does not assert a proof of the
Goldbach conjecture. Rather, it provides an equivalent topological
formulation of the conjecture. The arithmetic statement that every
even integer admits a Goldbach representation is transformed into
the existence of a common point in all singleton closures of the
relational space $(E_6,\tau_G)$.
\end{remark}

\begin{remark}
The role of the number $6$ is intrinsic to the construction. Indeed,
if $p$ occurs in a Goldbach representation of $m$, then $m\mathcal R_G(p+3)$,
while \(3\in S(p+3)\) implies $(p+3)\mathcal R_G6$. Thus every Goldbach
representation generates a relational path of
length at most two terminating at $6$.
\end{remark}

\bigskip

Having established the \(wR_0\) characterization, we now investigate how
the arithmetic structure of Goldbach representations is reflected in more
classical global topological properties. Connectedness has appeared before
in graph-theoretic formulations of Goldbach
\cite{DasGhoshGhoshSen2021}, although in a different setting from the
relational Alexandroff space considered here.  We first show that a Goldbach exception produces
a nontrivial clopen singleton, while the validity of Goldbach prevents such
a separation.

\begin{theorem}[Connectedness characterization]
\label{thm:goldbach-connected}
The following statements are equivalent:
\begin{enumerate}
    \item[\rm (i)] The strong Goldbach conjecture is true;
    
    \item[\rm (ii)] The Goldbach relational space
    $(E_6,\tau_G)$ is connected.
\end{enumerate}
\end{theorem}

\begin{proof}
Assume first that the strong Goldbach conjecture is true. Let
$A\subseteq E_6$ be a nonempty subset which is both open and closed.
Choose $m\in A$. Since the Goldbach conjecture holds, $S(m)\neq\varnothing$. By Lemma~\ref{lem:goldbach-to-six}, $m\leadsto_{\mathcal R_G}6$. Since $A$ is closed, Proposition~\ref{prop:closed-forward} implies that $A$ is forward invariant under $\mathcal R_G$. Hence $6\in A$. Suppose that $E_6\setminus A\neq\varnothing$. Since $A$ is open,
$E_6\setminus A$ is closed. Choose $n\in E_6\setminus A$. Again, by the Goldbach conjecture, $n\leadsto_{\mathcal R_G}6$. Since $E_6\setminus A$ is closed and therefore forward invariant,
we obtain $6\in E_6\setminus A$, contradicting $6\in A$. Hence $A=E_6$. Thus the only clopen subsets of $E_6$ are $\varnothing$ and $E_6$,
and consequently $(E_6,\tau_G)$ is connected. Conversely, suppose that the strong Goldbach conjecture is false.
Since $4=2+2$, there exists some $m_0\in E_6$ such that $S(m_0)=\varnothing$. By Proposition~\ref{prop:failure-goldbach-isolated}, $\overline{\{m_0\}}=\{m_0\}$,
so $\{m_0\}$ is closed. We now show that $\{m_0\}$ is also open. Suppose that there existed $n\in E_6$ such that $n\mathcal R_G m_0$. By the definition of $\mathcal R_G$, there would exist
$p\in S(n)$ such that $m_0=p+3$. Since both $p$ and $3$ are prime, this would give a Goldbach representation $m_0=p+3$, contradicting $S(m_0)=\varnothing$. Therefore $\mathcal R_G^{*}(\{m_0\})=\varnothing$. Hence
\[
\mathcal R_G^{*}(\{m_0\})
\subseteq
\{m_0\},
\]
and thus $\{m_0\}\in\tau_G$. Therefore $\{m_0\}$ is a nonempty proper clopen subset of $E_6$.
Consequently $(E_6,\tau_G)$ is disconnected.
\end{proof}

\bigskip

\begin{definition}
Define the set of Goldbach exceptions by
\[
\mathcal E_G
=
\{m\in E_6:S(m)=\varnothing\}.
\]
\end{definition}

\bigskip

Connectedness distinguishes between the validity and failure of the
conjecture, whereas compactness detects a weaker arithmetic phenomenon.
The next theorem shows that compactness records exactly whether the set of
possible Goldbach exceptions is finite.

\begin{theorem}[Compactness characterization]
\label{thm:goldbach-compact}
The Goldbach relational space $(E_6,\tau_G)$ is compact if and only if
the set
\[
\mathcal E_G
=
\{m\in E_6:S(m)=\varnothing\}
\]
is finite.
\end{theorem}

\begin{proof}
Assume first that $\mathcal E_G$ is finite. Let $\mathscr U=\{U_i\}_{i\in I}$ be an arbitrary open cover of $E_6$. Since $6\in E_6$, there exists
$i_0\in I$ such that $6\in U_{i_0}$. We claim that $U_{i_0}$ contains every point of
$E_6\setminus\mathcal E_G$. Let $m\in E_6\setminus\mathcal E_G$. Then
\[
S(m)\neq\varnothing,
\]
and by Lemma~\ref{lem:goldbach-to-six}, $m\leadsto_{\mathcal R_G}6$. Since $U_{i_0}$ is open and contains $6$, the defining property
\[
\mathcal R_G^{*}(U_{i_0})\subseteq U_{i_0}
\]
can be applied repeatedly along a finite $\mathcal R_G$-chain from
$m$ to $6$. Hence $m\in U_{i_0}$. Therefore
\[
E_6\setminus\mathcal E_G
\subseteq U_{i_0}.
\]
Since $\mathcal E_G$ is finite, for each
$c\in\mathcal E_G$ choose some member $U_{i(c)}$ of the cover such
that $c\in U_{i(c)}$. Then
\[
U_{i_0}
\cup
\bigcup_{c\in\mathcal E_G}U_{i(c)}
\]
is a finite subcover of $E_6$. Thus $(E_6,\tau_G)$ is compact. Conversely, suppose that $\mathcal E_G$ is infinite. We first note
that, for every $c\in\mathcal E_G$, the singleton $\{c\}$ is open.
Indeed, if there existed $n\in E_6$ satisfying $n\mathcal R_G c$, then for some $p\in S(n)$ we would have $c=p+3$,
which would express $c$ as the sum of two primes. This contradicts
$c\in\mathcal E_G$. Hence
\[
\mathcal R_G^{*}(\{c\})=\varnothing,
\]
and therefore $\{c\}\in\tau_G$. Now define
\[
G=E_6\setminus\mathcal E_G.
\]
We claim that $G$ is open. Suppose that $y\in G$ and $x\mathcal R_Gy$. Then $x$ has an $\mathcal R_G$-successor. By the definition of the
Goldbach relation this implies $S(x)\neq\varnothing$, and hence $x\in G$. Therefore
\[
\mathcal R_G^{*}(G)\subseteq G,
\]
so that $G\in\tau_G$. Consequently,
\[
\mathscr V
=
\{G\}
\cup
\bigl\{\{c\}:c\in\mathcal E_G\bigr\}
\]
is an open cover of $E_6$. Since $\mathcal E_G$ is infinite and each
$c\in\mathcal E_G$ occurs only in its corresponding singleton in this
cover, no finite subfamily of $\mathscr V$ can cover all of
$\mathcal E_G$. Thus $\mathscr V$ has no finite subcover, and
$(E_6,\tau_G)$ is not compact.
\end{proof}

\bigskip

The previous characterization immediately yields a topological consequence
of the strong Goldbach conjecture. If no exceptions exist, then the
exception set is certainly finite, and the Goldbach relational space must
therefore be compact.

\begin{corollary}
\label{cor:goldbach-compact}
If the strong Goldbach conjecture is true, then
$(E_6,\tau_G)$ is compact.
\end{corollary}

\begin{proof}
If the Goldbach conjecture is true, then
\[
\mathcal E_G=\varnothing.
\]
Hence $\mathcal E_G$ is finite, and the conclusion follows from
Theorem~\ref{thm:goldbach-compact}.
\end{proof}

\bigskip

For the present relational space, compactness can be strengthened without
changing the underlying arithmetic condition. The finite decomposition
induced by the exception set allows compactness to be identified with
supercompactness, leading to the following refinement.

\begin{corollary}
\label{cor:goldbach-supercompact}
The following statements are equivalent:
\begin{enumerate}
    \item[\rm (i)] $(E_6,\tau_G)$ is compact;
    \item[\rm (ii)] $(E_6,\tau_G)$ is supercompact;
    \item[\rm (iii)] the set of Goldbach exceptions
    \[
    \mathcal E_G=\{m\in E_6:S(m)=\varnothing\}
    \]
    is finite.
\end{enumerate}
In particular, if the strong Goldbach conjecture is true, then
$(E_6,\tau_G)$ is supercompact.
\end{corollary}

\begin{proof}
By Theorem~\ref{thm:goldbach-compact},
$(E_6,\tau_G)$ is compact, if and only if $|\mathcal E_G|<\infty$. If $\mathcal E_G$ is finite, then
\[
E_6=
\bigl(E_6\setminus\mathcal E_G\bigr)
\sqcup
\bigcup_{c\in\mathcal E_G}\{c\}.
\]
The subspace $E_6\setminus\mathcal E_G$ is supercompact, since every
nonempty closed subset contains $6$, and each singleton $\{c\}$ is
supercompact. Since a finite topological sum of supercompact spaces is
supercompact, $(E_6,\tau_G)$ is supercompact. Conversely, every supercompact space is compact. 
\end{proof}

\section{Conclusions and future directions}

In this paper we introduced a relational Alexandroff topology
\(\tau_G\) on $E_6=\{6,8,10,\ldots\}$ associated with Goldbach representations. The construction is based on
the multivalued map
\[
F_G(m)=\{p+3:p\in S(m)\},
\]
where
\[
S(m)=\{p\in\mathbb P:m-p\in\mathbb P\}.
\]
The corresponding relational structure provides a natural
reachability interpretation of singleton closures and singles out the
point \(6\) as a global terminal point whenever a Goldbach
representation exists. The main result establishes that the strong Goldbach conjecture is
equivalent to
\[
\bigcap_{m\in E_6}\overline{\{m\}}=\{6\},
\]
or, equivalently, to the fact that the Goldbach relational space
\((E_6,\tau_G)\) is not \(wR_0\). We have also shown that the
conjecture is equivalent to connectedness of the Goldbach space,
whereas compactness is equivalent to the existence of at most finitely
many Goldbach exceptions. Moreover, for this particular relational
space, compactness and supercompactness coincide.

These results suggest several natural directions for further research.
The first concerns an algebraic reformulation of the Goldbach
conjecture through the semiring of open sets. Any topology
\(\tau\) can be regarded as a commutative idempotent semiring by taking
union as addition and intersection as multiplication. Guale and
Vielma \cite{GualeVielma2026} showed that a topological space is not
\(wR_0\) if and only if there exists a nonempty closed set \(F\)
whose avoidance ideal
\[
\Phi(F)=\{U\in\tau:U\cap F=\varnothing\}
\]
is the unique maximal ideal of the topology semiring. They applied
this principle to the primal Collatz space and obtained an algebraic
characterization of the Collatz conjecture in terms of local
semirings. In the present setting, the main characterization theorem identifies
the common closed core explicitly as
\[
K_{\mathcal R_G}(E_6)=\{6\}.
\]
Consequently, the results of \cite{GualeVielma2026} strongly suggest
the following algebraic formulation:
\[
\text{the strong Goldbach conjecture holds}
\]
if and only if the topology semiring \(\tau_G\) has a unique maximal
ideal, namely the avoidance ideal
\[
\Phi(\{6\})
=
\{U\in\tau_G:6\notin U\}.
\]
Thus the topological characterization developed here opens a direct
route toward a purely algebraic characterization of Goldbach in terms
of local semirings. A systematic study of the ideal structure of
\(\tau_G\), including its maximal, prime, and principal ideals, appears
to be a natural continuation of the present work.

A second direction concerns the prime spectrum of the Goldbach
topology semiring. Ruza and Vielma \cite{RuzaVielma2009} studied
Gelfand semirings and the Zariski topology on their prime spectra.
Recall that a commutative semiring \(R\) is called a Gelfand semiring
when every prime ideal is contained in a unique maximal ideal. They
proved, among other results, that for a Gelfand semiring \(R\),
locality of \(R\) is equivalent to connectedness of its prime spectrum
equipped with the appropriate Alexandroff--Zariski topology.

This suggests investigating whether the Goldbach topology semiring
\(\tau_G\) belongs to the class of Gelfand semirings. If this is the
case, the locality characterization arising from the \(wR_0\)
property could potentially be transferred to the spectrum
\[
\operatorname{Spec}(\tau_G),
\]
leading to a spectral characterization of the Goldbach conjecture in
terms of connectedness of a Zariski-type space.

The supercompactness obtained for the Goldbach relational space also
suggests a further connection with the theory developed in
\cite{RuzaVielma2009}. In that work, supercompactness properties of
topologies on the prime spectrum are related to uniqueness of minimal
prime ideals. It would therefore be interesting to determine whether
the supercompactness of \((E_6,\tau_G)\) can be transferred, through
the topology semiring, to an appropriate topology on
\(\operatorname{Spec}(\tau_G)\). Such a result could provide a second
algebraic--topological characterization of the Goldbach conjecture,
now involving Gelfand or \(m\)-semiring structures and the geometry of
the prime spectrum.

However, this spectral program requires additional work. In
particular, one must determine whether \(\tau_G\) is a Gelfand
semiring, describe its prime and maximal ideals, and establish the
precise relation between supercompactness of the original Goldbach
space and the relevant Zariski or Alexandroff topology on its spectrum.
These questions appear to form a natural next stage of the theory. The relational construction considered here suggests
that arithmetic conjectures admitting a global reachability
formulation may also admit associated Alexandroff spaces, semiring
structures, and spectral representations. Exploring this interaction
between number theory, general topology, and semiring theory may
provide a useful framework for further topological and algebraic
reformulations of classical arithmetic problems.

\end{document}